\documentclass[11pt]{amsart}

\usepackage{amscd,amssymb,amsopn,amsmath,amsthm,mathrsfs,graphics,amsfonts,enumerate,verbatim,calc
}
\usepackage[all,cmtip]{xy}
\usepackage[all]{xy}
\usepackage{tikz}
\usepackage{lscape}
\usetikzlibrary{matrix,arrows,decorations.pathmorphing}
\usetikzlibrary {positioning}
\definecolor {processblue}{cmyk}{0.96,0,0,0}
\usepackage{hyperref}

\usepackage{caption} 
\usepackage[utf8]{inputenc}

\usepackage{cite}

\usepackage{color}

\usepackage[OT2,OT1]{fontenc}
\newcommand\cyr{%
\renewcommand\rmdefault{wncyr}%
\renewcommand\sfdefault{wncyss}%
\renewcommand\encodingdefault{OT2}%
\normalfont
\selectfont}
\DeclareTextFontCommand{\textcyr}{\cyr}

\usepackage{amssymb,amsmath}

\DeclareFontFamily{OT1}{rsfs}{}
\DeclareFontShape{OT1}{rsfs}{n}{it}{<-> rsfs10}{}
\DeclareMathAlphabet{\mathscr}{OT1}{rsfs}{n}{it}

\numberwithin{equation}{section}
\newtheorem{theorem}{Theorem}[section]
\newtheorem{lemma}[theorem]{Lemma}
\newtheorem{prop}[theorem]{Proposition}
\newtheorem{cor}[theorem]{Corollary}

\theoremstyle{definition}
\newtheorem{defn}[theorem]{Definition}
\newtheorem{remark}[theorem]{Remark}
\theoremstyle{remark}

\newtheorem{example}[theorem]{Example}

\newcommand{\degree}{\operatorname{degree}}

\newcommand{\Spec}{\operatorname{Spec}}

\newcommand{\diam}{\operatorname{diam}}
\newcommand{\dist}{\operatorname{dist}}
\newcommand{\height}{\operatorname{ht}}

\newcommand{\Ext}{\operatorname{Ext}}
\newcommand{\Supp}{\operatorname{Supp}}

\newcommand{\codim}{\operatorname{codim}}
\newcommand{\depth}{\operatorname{depth}}

\begin{document}
\title[On the diameter of dual graphs of $(S_2)$ Stanley-Reisner rings]{On the diameter of dual graphs of Stanley-Reisner rings with Serre $(S_2)$ property and Hirsch type bounds on abstractions of polytopes}

\author{Brent Holmes}
\address{Department of Mathematics\\
University of Kansas\\
 Lawrence, KS 66045-7523 USA}
\email{brentholmes@ku.edu}
\date{\today}



\begin{abstract} Let $R$ be an equidimensional commutative Noetherian ring of positive dimension.  The dual graph of $R$ is defined as follows: the vertices are the minimal prime ideals of $R$, and the edges are the pairs of prime ideals $(P_1,P_2)$ with height $(P_1 + P_2) = 1$.  If $R$ satisfies Serre's property $(S_2)$, then $\mathcal{G} (R)$ is connected.  In this note, we provide lower and upper bounds for the maximum diameter of dual graphs of Stanley-Reisner rings satisfying $(S_2)$.  These bounds depend on the number of variables and the dimension.  Dual graphs of $(S_2)$ Stanley-Reisner rings are a natural abstraction of the $1$-skeletons of polyhedra.  We discuss how our bounds imply new Hirsch-type bounds on $1$-skeletons of polyhedra.
\end{abstract}

\maketitle

\section{Introduction}

The polynomial Hirsch conjecture states that a $d$-dimensional polyhedron with $n$ facets has diameter bounded above by a polynomial expression in $n-d$.  The diameter of a polyhedron is the diameter of its 1-skeleton.  The polynomial Hirsch conjecture is a weakening of the Hirsch conjecture, which was disproved by Klee and Walkup \cite{KW67} in the general case and Santos \cite{Sa11} in the bounded case.  For a history of the Hirsch and polynomial Hirsch conjectures, see \cite{Sa13}.  In this paper we construct bounds which improve on bounds from existing literature \cite{La70, Ba74, EH10}, but are not polynomial.  Our bounds are sharp for small $n$ and $d$.  Many authors have examined  the diameters of generalizations of polyhedra (e.g.~\cite{AD74, CS16, EH10, Ka92}).  We consider generalizations of polyhedra whose 1-skeletons $\mathcal{G}$ have vertices that are subsets of size $d$ of $\{ 1,2,...,n \}$, such that $\mathcal{G}$ has the following properties (see Section 1 of~\cite{EH10}):  

\begin{enumerate}

\item (i) For each $u, v \in V(G)$ there exists a path connecting $u$ and $v$ whose intermediate vertices all contain 
$u  \cap  v$.

\item (ii) The edge $(u, v)$ is present if and only if $|u \cap v| = d - 1$.

\end{enumerate}

Generalized polyhedra of this type have been considered in section $4.1$ of~\cite{Ka92}.  

Dual graphs are an object of wide interest in commutative algebra and algebraic geometry (e.g. \cite{Ha62, BB15, BV15,  BM17, NS17}).  From our setting, we shall consider the dual graph to have vertices corresponding to the minimal primes of a ring, however, the dual graph can be constructed in a more general setting with vertices corresponding to the irreducible components of a scheme.  It is a famous result of Hartshorne \cite{Ha62} that if $X$ is a connected projective scheme such that $\mathcal{O}_{X,x}$ satisfies Serre’s condition $(S_2)$ for all $x \in X$, then the dual graph of $X$ is connected.  This result is commonly known in its less general form to say arithmetically Cohen–Macaulay projective schemes have connected dual graphs.

Stanley-Reisner rings satisfying $(S_2)$ have recently attracted much attention \cite{ MT09, HT11, PF14, DH16}.  It is known that a graph having properties (i) and (ii) is equivalent to that graph being the dual graph of a Stanley-Reisner ring satisfying Serre's condition $(S_2)$.  We shall combine techniques from commutative algebra and combinatorics to prove bounds on the diameter of these graphs.

We define $\mu (d,n)$ to be the largest diameter of a dual graph of an $(S_2)$ Stanley-Reisner ring of dimension $d$ and codimension $n-d$.  One of the main results of this paper is the determination of the precise values of $\mu (d,n)$ for small $n$ and $d$ (see Table 1).  

To that end we first construct upper bounds for quite general $n$ and $d$.  For instance, Theorem \ref{dimthree} shows $\mu (3,n) \leq \max (2n-10,n-2)$.  In Theorem \ref{upperboundallnd}, we prove $\mu (d,n) \leq 2^{d-2}(n-d)$, which improves on the bound of~\cite{EH10} (See Remark \ref{upperboundremark}).  In Theorem \ref{bigupperbound}, we prove $\mu (d,n) \leq 3 \cdot 2^{\frac{n-d-5}{2}}(n-d)$.  This result is derived using our bound from Theorem \ref{upperboundallnd}.  Combining these results with manually generated constructions (Appendix \ref{AppendixA}), we can produce the following table of exact values of $\mu(d,n)$:

\begin{table}[h!]
\caption{$\mu(d,n)$ for small $n$ and $d$}
\label{table:1}
\centering
\begin{tabular}{ |p{1.5cm}|p{1.5cm}|p{1.5cm}|p{1.5cm}|  }
\hline
 & d=2 & d=3 & d=4 \\
\hline
n=6 & 4 & 4 & 2 \\
\hline
n=7 & 5 & 5 & 3  \\
\hline
n=8 & 6 & 6 & 6 \\
\hline
n=9 & 7 & 7 & 7 \\
\hline
n=10+ & $n-2$ & $\geq n-1$ & $\geq n-2$ \\
\hline
\end{tabular}
\end{table}

In Section 6 of~\cite{Sa11}, Santos builds arbitrarily large complexes whose diameters exceed the Hirsch bound by a fixed fraction.  This is achieved by  using a \textit{gluing lemma} from~\cite{HK98}, which states that two $d$-dimensional polytopes, $P_1$, $P_2$, can be glued together yielding a new polytope $P$ with $\diam P \geq \diam P_1 + \diam P_2 -1$.  We present Theorem \ref{gluing}, an algebraic analogue to the gluing lemma.  This theorem tells us that two $d-1$-dimensional complexes with $(S_\ell)$ rings (we shall call these $(S_\ell)$ complexes) glued together along a pure, $(S_{\ell - 1})$ subcomplex of dimension at least $d-2$ yield an $(S_\ell)$ complex.  The proof of Theorem \ref{gluing} is achieved using local cohomology and local duality.  

Applying Theorem \ref{gluing}, we are able to construct complexes whose Stanley-Reisner rings have dual graphs with arbitrarily large diameter which (with proper labeling) have properties (i) and (ii).  For appropriate complexes $\Delta$ and $\Delta '$, \[ \diam G(k[\Delta ]) + \diam G(k[\Delta ']) =\diam G(k[\Delta \cup \Delta ']). \]  Gluing multiple copies of examples from the small $n$ and $d$ cases together, we construct graphs with properties (i) and (ii) in dimensions $3$ and $4$ with diameters $\frac{5}{4}(n-d)$ and $\frac{3}{2}(n-d)$ respectively (see Theorem \ref{gluingd3}, Theorem \ref{gluingd4}).  We show graphs with properties (i) and (ii) and diameter $\frac{3}{2}(n-d)$ can be constructed for all $d \geq 4$ (see Remark \ref{gluingremark}).  

Introduction of terms is covered in Section \ref{intro}.  In Section \ref{s2}, we demonstrate that a graph having properties (i) and (ii) is equivalent to that graph being the dual graph of a Stanley-Reisner ring satisfying Serre's condition $(S_2)$.  In Section \ref{upper}, we prove the upper bounds introduced earlier in this  section.  Details of the process of gluing to preserve $(S_2)$ are discussed in Section \ref{glue}, and constructions of glued complexes are discussed in Section \ref{glueconstruct} with examples displayed in Figures \ref{4dfigure} and \ref{3dfigure}.  In Appendix \ref{AppendixA}, we show the constructions needed which justify Table \ref{table:1} and investigate the relationship between $(S_2)$ and Buchsbaum.

\section{Background and Notation}

\subsection{Introduction of Terms}\label{intro}

\begin{defn}
A d-dimensional \textit{polyhedron} is a non-empty intersection of finitely many closed half spaces of $\mathbb{R}^d$.
\end{defn}

\begin{defn}
A \textit{facet} of a simplicial complex $\Delta$ is a simplex of $\Delta$ which is not properly contained in another simplex of $\Delta$.
\end{defn}

\begin{defn}
The \textit{$1$-skeleton of a polyhedron} is the set of vertices and edges of the polyhedron.
\end{defn}

A polyhedron whose vertices are the intersection of $d$ of its facets is non-degenerate.  Any polyhedron can be transformed into a non-degenerate polyhedron by perturbation without decreasing its diameter~\cite{EH10}.  Therefore, we may restrict our attention to non-degenerate polyhedra.  

\begin{defn}
A \textit{pure} simplicial complex is a simplicial complex whose facets all have the same dimension.
\end{defn}

We remind the reader that a $(d-1)$-dimensional simplicial complex has a $d$-dimensional Stanley-Reisner ring.  We will use as notation $\Delta_R$ to be the simplicial complex with Stanley-Reisner ring $R$.  We shall use $\Delta$ when the ring is either unspecified or clear from context.

Let $k$ be a field and $S = k[x_1,...,x_n]$.  Let $\Delta$ be a pure, $(d-1)$-dimensional simplicial complex with Stanley-Reisner ring $R = S/I$, where $I$ is the intersection of the minimal prime ideals $P_i$ of $R$.  Let $\{ F_i \}$ be the facets of $\Delta$.  With proper numbering, $P_i$ is generated by $\{ x_j | x_j \notin F_i \}$ (see e.g. the survey ~\cite{FM14} for proof).  Thus purity of $\Delta$ is equivalent to each $P_i$ being generated by $n - d$ distinct variables.  

\begin{defn}
Let $\mathcal{G} (R)$ be the graph with $V({\mathcal{G}} (R)) = \{v_i = \Pi x_j \}$ where the $x_j$'s generate $P_i$, $E({\mathcal{G}} (R)) = \{ (v_j, v_k) | \height (P_j + P_k) = 1 \}$.  Then $\mathcal{G} (R)$ is the \textit{dual graph} of $R$.
\end{defn}

This type of graph is often constructed in a more general setting applying to schemes (e.g.~\cite{BB15}).  This definition follows the definition of Hochster and Huneke \cite{HH94} and is equivalent to other definitions, (e.g. \cite{BB15}).  In this paper we consider dual graphs of Stanley-Reisner rings.  It should be noted that not every graph is a dual graph of a Stanley-Reisner ring \cite{BV15}).

\begin{defn}
Let $\Delta$ be a simplicial complex on $\{ 1,2, \cdots , n \}$.  The \textit{Alexander dual} of $\Delta$ is \[ \Delta ^\vee = \{ F \subseteq \{ 1,2, \cdots , n \} | \{ 1,2, \cdots , n \} \backslash F \notin \Delta \}. \]
\end{defn}

Let $S/I$ be the Stanley-Reisner ring of $\Delta$, $S/I^\vee$ be the Stanley-Reisner ring of $\Delta ^\vee$.  We refer to $I^\vee$ as the Alexander dual of $I$.  The Alexander dual of $I$ is generated by the product of the generators of each minimal prime ideal of $I$ (see e.g. the survey ~\cite{FM14} for proof).

\begin{example}

Let $S = k[x_1,x_2,x_3,x_4,x_5,x_6]$ and \[ I= \langle x_1x_3x_5, x_1x_3x_6,x_1x_4x_5, x_1x_4x_6,x_2x_3x_5, x_2x_3x_6,x_2x_4x_5, x_2x_4x_6 \rangle = \langle x_1, x_2 \rangle \cap \langle x_3, x_4 \rangle \cap \langle x_5,x_6 \rangle. \]
Then $I^\vee = \langle x_1x_2, x_3x_4,x_5x_6 \rangle$.

\end{example}

\begin{remark}
We notice that the vertices of $G(S/I)$ are in one to one correspondence with the generators of $I^\vee$.  Also, each vertex is comprised of $n-d$ variables.  Finally, notice that by definition every dual graph has property (ii).
\end{remark}

\begin{defn}
Given a graph $\mathcal{G} (R)$, we define $\bar{\mathcal{G}} (R)$, to be the graph with $V(\bar {\mathcal{G}} (R)) = \{ \bar v_i = \frac{ x_1 \cdot x_2 \cdot ... \cdot x_n}{v_i}\}$ where $\{ v_i \} = V({\mathcal{G}} (R))$ and $E(\bar {\mathcal{G}} (R)) = \{ \bar v_i, \bar v_j \}$ where $E({\mathcal{G}} (R)) = \{ v_i,v_j \}$. 
\end{defn}

The graph $\bar{\mathcal{G}} (R)$ is a relabeling of $\mathcal{G} (R)$.  We construct this labeling so that our graphs fit the setting of \cite{EH10}, and so that we can determine if our graphs have properties (i) and (ii).

\begin{remark}
We note $\bar{\mathcal{G}} (R)$ is the facet-ridge graph of the complex with Stanley-Reisner ring $R$.
\end{remark}

\begin{defn}
A graph $\bar{\mathcal{G}} (R)$ is \textit{locally connected} if for any two vertices $\bar v_i, \bar v_j \in \bar {\mathcal{G}} (R)$, there exists a path between them such that each vertex in the path contains $\bar v_i \cap \bar v_j$.
\end{defn}

Locally connected is also referred to as ultra connected in~\cite{Ka92}.

\begin{remark}
Locally connected graphs are commonly referred to as \textit{normal} graphs in combinatorial circles.  Normality of a graph is a notion unrelated to normality of a ring.  To avoid this confusion, we use the name ``locally connected," which has been motivated by Theorem \ref{locally}, which connects this property of $\bar{\mathcal{G}} (R)$ to localization of the ring $R$.
\end{remark}


\subsection{The $(S_{\ell})$ condition}\label{s2}

We have shown graphs with properties (i) and (ii) from \cite{EH10} are an abstraction of $1$-skeletons of non-degenerate polyhedra.  Ensuring a graph satisfies property (ii) is not difficult.  Using Serre's condition and syzygy matrices, we demonstrate a simple method to ensure a dual graph of an $(S_2)$ Stanley-Reisner ring satisfies property (i).  


\begin{defn}
A ring satisfies \textit{Serre's condition $(S_{\ell})$} if for all $P$ in $\Spec R$, \[ \depth R_P \geq \min \{ {\ell}, \dim R_P \}. \]
\end{defn}

\begin{defn}
Let $M$ be an $R$ module with minimal generating set $\{ z_1,z_2,...,z_k \}$.  A \textit{first syzygy} of $M$ is a non-zero vector $(a_1,...,a_k) \in R^k$ such that $a_1z_1+...+a_kz_k = 0$.  A \textit{first syzygy matrix} of a module is a matrix whose columns span all the first syzygies of that module.
\end{defn}

\begin{theorem}[Yanagawa]\label{Yana}
The Stanley-Reisner ring $R=S/I$ satisfies $(S_2)$ if and only if $R$ is equidimensional and $I^\vee$ has a first syzygy matrix with only linear entries. 
\end{theorem}

This theorem is true by Corollary $3.7$ in~\cite{Ya00}.  Since checking the linearity of a first syzygy matrix can be easily done computationally, this theorem provides us with a simple way to demonstrate if a Stanley-Reisner ring satisfies $(S_2)$.  Now we connect $(S_2)$ to our locally connected condition.  We will use local duality, the ext module, and the cohomology module.  For background on these topics, see a Homological Algebra Text (e.g. \cite{We94}). 

\begin{theorem}\label{locally}
Let $R$ be an equidimensional Stanley-Reisner ring.  The following are equivalent:
\begin{enumerate}

\item $R$ satisfies $(S_2)$.

\item For any prime ideal $P$ generated by variables, ${\mathcal{G}} (R_P)$ is connected.

\item $\bar{\mathcal{G}} (R)$ is locally connected.

\end{enumerate}
\end{theorem}

\begin{proof}
$(2) \Leftrightarrow (3)$:  
Let $S$ be a subset of $\{ x_1,...x_n \}$.  Let $\bar {\mathcal{G}} (R)_S$ be the induced subgraph of $\bar{\mathcal{G}} (R)$ with 
\[ V(\bar {\mathcal{G}} (R)_S) = \{ v \in V(\bar {\mathcal{G}} (R)) | x_i \in v \ \textrm{ for all } x_i \in S \}. \]

$\bar{\mathcal{G}} (R)$ is locally connected if and only if $\bar {\mathcal{G}} (R)_S$ is connected for any choice of $S$.  By the definition of $\bar{\mathcal{G}} (R)$, $\bar {\mathcal{G}} (R)_S$ is a relabeling of a subgraph of $\mathcal{G} (R)$ whose vertices are the minimal prime ideals of $R$ contained in $P$, the prime generated by $\{ x_j | x_j \notin S\}$.  This subgraph is ${\mathcal{G}} (R_P)$.  Therefore ${\mathcal{G}} (R_P)$ connected for all primes generated by variables is equivalent to $\bar{\mathcal{G}} (R)$ being locally connected.

$(1) \Rightarrow (2)$:
Suppose ${\mathcal{G}} (R_P)$ is not connected for some nonempty set of primes generated by variables in $\Spec R$.  Note that every prime in this set must have height at least $2$.  Let us choose any prime $P$ maximal in this set.  If ${\mathcal{G}} (R_P)$ is not connected then $\bar {\mathcal{G}} (R_P)$ is not connected.  Suppose $\bar {\mathcal{G}} (R_P)$ contains two distinct components each of which contain $x_i \in \Supp R_P$.  Then $\bar {\mathcal{G}} (R_P)$ can be localized at $S=x_i$ yielding a disconnected graph, contradicting the maximality of $P$.  Therefore, each connected component of $\bar {\mathcal{G}} (R_P)$ is composed of disjoint sets of variables.  The vertices of $\bar {\mathcal{G}} (R_P)$, however, represent the facets of the simplicial complex $\Delta_{R_P}$.  Therefore, we have that $\Delta_{R_P}$ is not connected, which implies $H_{PR_P}^1(R_P) \neq 0$.  Thus $\depth R_P \leq 1$; however, $\height P \geq 2$, and thus $\dim R_P \geq 2$.  Therefore, $R$ is not $(S_2)$. 

$(1) \Leftarrow (2)$:
Using local duality, we have a ring $R$ satisfies $(S_2)$ if and only if $\dim \Ext _S^{n-i}(R,\omega_S) \leq i-2$ for all $i < d$ (see Lemma \ref{gluelemma}).  From~\cite{Ya00}, since $R$ is a Stanley-Reisner ring, $\Ext _S^{n-i}(R,\omega_S)$ is a square-free module.  Thus $\Ext _S^{n-i}(R,\omega_S)$ is uniquely determined by its primes generated by variables.  The dimension of $\Ext _S^{n-i}(R,\omega_S)$ determines if $R$ satisfies $(S_2)$.  Therefore, we only need consider primes generated by variables when showing that $R$ satisfies $(S_2)$.

If ${\mathcal{G}} (R_P)$ is connected for all primes generated by variables in $\Spec R$, then $\Delta_{R_P}$ is connected for all primes generated by variables with height at least $2$ in $\Spec R$.  When $\height P \geq 2$, $\Delta_{R_P}$ is connected if and only if $H_{PR_P}^1(R_P) = 0$.  Further, $H_{PR_P}^0(R_P) = 0$ for all $P \in \Spec R$, since $R$ is a Stanley-Reisner ring.  Thus $\depth R_P \geq 2$ for all primes generated by variables with height at least $2$, and $\depth R_P \geq 1$ for all primes generated by variables with height $1$.  Thus $R$ satisfies $(S_2)$.
\end{proof}
\begin{cor} Given a graph $\mathcal{G}$ with each vertex labeled with the same number of variables the following are equivalent:

\begin{enumerate}

\item $\mathcal{G}$ has properties (i) and (ii) from~\cite{EH10}.

\item $\mathcal{G}$ is $\bar{\mathcal{G}} (R)$ with $R$ an equidimensional Stanley-Reisner ring, and $\mathcal{G}$ has local connectedness.

\item $\mathcal{G}$ is $\bar{\mathcal{G}} (R)$ with $R$ an $(S_2)$ Stanley-Reisner ring.

\item $\mathcal{G}$ is $\bar{\mathcal{G}} (R)$ with $R=S/I$ an equidimensional Stanley-Reisner ring such that the Alexander dual of $I$, $I^\vee$, has first syzygy matrix with all linear entries.

\end{enumerate}
\end{cor}

\begin{proof}
$(1) \Rightarrow (2)$:
Let each vertex label of $\mathcal{G}$ be a facet of a pure simplicial complex $\Delta$.  Let the Stanley-Reisner ring of $\Delta$ be $R$.  Then $\mathcal{G}$ has the same vertex set as $\bar{\mathcal{G}} (R)$.  Property (ii) implies that $\mathcal{G}$ has the same edge set as the dual graph of $R$.  Property (i) is equivalent to locally connected. 

$(1) \Leftarrow (2)$:
Property (ii) is required by the definition of a dual graph.  Property (i) is precisely the same as local connectedness.

$(2) \Leftrightarrow (3)$:
A ring with property $(S_2)$ is equidimensional.  Thus, by Theorem \ref{locally}, $\bar{\mathcal{G}} (R)$ is locally connected if and only if $R$ satisfies $(S_2)$.

$(3) \Leftrightarrow (4)$:
See Theorem \ref{Yana}.
\end{proof}

\section{Upper Bounds}\label{upper}
In this section we prove upper bounds for the diameters of dual graphs of $(S_2)$ Stanley-Reisner rings.  This is achieved by working with $\bar{\mathcal{G}} (R)$ (see Section \ref{intro}).

\begin{defn}
A \textit{strictly increasing path} is a path such that each subsequent vertex in the path has larger distance from the starting vertex of the path.
\end{defn}

\begin{theorem}\label{dimthree}
$\mu (3,n) \leq \max (2n-10,n-2)$.
\end{theorem}

\begin{proof}
We construct this upper bound in a manner inspired by~\cite{EH10}.  Let $R$ be an $(S_2)$ Stanley-Reisner ring, such that $\bar{\mathcal{G}} (R)$ has vertices $ABC,DEF$ with maximum distance in the graph.  Let us assign to each vertex $v$ the integer $\dist (ABC,v)$, where $\dist (ABC, v)$ denotes the length of the shortest path from $ABC$ to $v$.  We define layers $L_i = \{ v \in V(\bar {\mathcal{G}} (R)) | \dist (ABC,v) = i \}$.  

Define a block of layers to be a set of layers $\{ L_i | a \leq i \leq b \}$ for some integers $a,b$.  Let $c$ be the largest integer such that $L_c$ contains $A,B,$ or $C$.  Let $B_1$ = $\{ L_i | 0 \leq i \leq c \}$.  Let $n_0$ be the number of variables in $B_1$.  Without loss of generality, $A$ is contained in $L_c$.  By local connectedness, there exists a path consisting of vertices which all contain $A$ from $L_0$ to $L_c$.  This path can have maximum length $n_0-3$ (see the $d=2$, $n=n_0-1$ case in Appendix \ref{AppendixA}).

Next, we construct a second block.  Let $d$ be the largest integer such that $L_d$ contains a variable of $L_c$; call this variable $a_1$.  Let $B_2 = \{L_i | c < i \leq d \}$.  Let $n_1$ be the number of variables in $B_2$ but not in $B_1$.  The diameter of this layer will be bounded by the maximum length of a path in which each vertex contains $a_1$.  This path will have maximum length of $n_1 + n_0 -3 -3$ (the second $-3$ is to account for the fact that $A,B,C$ cannot be in the layers of this block).

Construct the third block $B_3$ in the same way.  Its longest path will have maximal length $n_2 + n_1 -3 -3$.  By construction, $B_3$ cannot have any elements in common with $B_1$.  Also, $B_3$ does not contain any variables in the $c+1$ layer (there are at least $3$ such variables).

Continue in this manner.

We sum the lengths of the blocks and add $1$ for each path between blocks to obtain:

$2n_0 + 2n_1 + \cdots + 2n_{k-2} + n_{k-1} - 5(k-2) - 8$ where $k$ is the number of blocks.

$k \geq 3$ implies $\diam \bar {\mathcal{G}} (R) \leq 2n-13$.

$k=1$ implies $n_0=n$ and $\diam \bar {\mathcal{G}} (R) \leq n-3$.

Let $k=2$.  If $ADE$, $ADF$, or $AEF$ is a vertex, then our path containing $A$ has length at most $n-3$, and some vertex in that path is adjacent to $DEF$.  Therefore, $\diam \bar {\mathcal{G}} (R) \leq n-2$.  

Let $L_j$ be the largest layer containing the vertex $x_i$.  Define $L_{x_i} = j$.  Let $L_A \geq L_B \geq L_C$.  Let $i^*= \min \{ i | AD \subseteq v \in L_i \} \leq \min \{ i | AE \subseteq v \in L_i \} \leq \min \{ i | AF \subseteq v \in L_i \}$.  Denote the vertex containing $AD$ in $L_{i^*}$ to be $v^*$ (if more than one exists choose any such one).  

Consider the case $L_A = n-3$.  To maximize $i^*$, we must have a path with tail: \[ Ax_iD, Ax_jD, Ax_jE, Ax_kE, Ax_kF. \]  Thus the path from $ABC$ to $v^*$ must have length at most $n-7$.   The maximum length from $v^*$ to $DEF$ is $n-3$.  Thus, $L_{D} \leq 2n-10$.

Consider the case where $L_A = n-4$.  Take $v \in L_{n-4}$ such that $A \in v$ and construct a path from $ABC$ to $v$.  If the path contains $AD, AE$, and $AF$, then the path from $ABC$ to $v^*$ must have length at most $n-7$.  If this path does not contain one of those, say $AF$, then $i^* \leq n-6$.  If $i^* = n-6$, then $L_A \geq L_B \geq L_C$ implies any strictly increasing path from $v^*$ to $DEF$ of vertices all containing $D$ cannot have both a vertex containing $B$ and a vertex containing $C$.  Thus $L_{D} \leq i^*+n-4=2n-10$.  Consider $i^* < n-6$.  Any path from $v^*$ to $DEF$ will be at most length $n-3$.  Thus $L_{D} \leq 2n-10$.

Now suppose $L_A \leq n-5$.  Then $i^* \leq n-5$.  Further, $B,C \notin L_j$ for all $j > n-5$.  Thus any path from $L_A$ to $DEF$ in which each vertex contains $D$, will have maximal length $n-3-2$.  Thus $L_{D} \leq 2n-10$.

\end{proof}

\begin{cor} \label{C3}
$\mu (4,8) \leq 6$.
\end{cor}

\begin{proof} 
Let us consider $v_1, v_2 \in \bar {\mathcal{G}} (R)$.  If $v_1 \cap v_2 \neq \emptyset $ then we can reduce to the $n=7$, $d=3$ case, and thus $\dist (v_1,v_2) \leq 5$.  Thus  assume $v_1 \cap v_2 = \emptyset$.  Connectivity implies there exists $v_3$ adjacent to $v_1$.  The vertex $v_3$ must have a non-trivial intersection with $v_2$, thus $\dist (v_2, v_3) \leq 5$.  Therefore, $\dist (v_1,v_2) \leq 6$.  
\end{proof}

\begin{cor}
$\mu (d, d+4) = 6$.
\end{cor}

\begin{proof}
Let $R$ be a codimension-$4$, $(S_2)$ Stanley-Reisner ring.  Take $v_1, v_2 \in V(\bar {\mathcal{G}} (R))$.  Then, $v_1, v_2$ will each contain $n-4$ variables, and $v_1 \cap v_2$ will contain at least $n-8$ variables.  Thus there must be a path from $v_1$ to $v_2$ in which each vertex in that path contains those $n-8$ shared variables.  Thus $\mu (d,d+4) \leq \mu (4,8)$.  Furthermore, we may take the graph in Figure \ref{dim4} (see Appendix \ref{AppendixA}) and add the same $d-4$ variables to each vertex to show $\mu (d,d+4) \geq 6$.
\end{proof}

Let us now construct bounds for any values of $n$ and $d$.  

\begin{theorem}\label{upperboundallnd}
$\mu (d,n) \leq 2^{d-2}(n-d)$, for all $n$ and all $d \geq 2$.
\end{theorem}

\begin{proof}
Applying Theorem \ref{dimthree}, $\mu (3,n) \leq 2n-6$ for all $n \geq d$.  Thus the $d=3$ case holds.  We begin induction on $d$.  Let us partition the vertices of $\bar{\mathcal{G}} (R)$ into layers and blocks, as in the proof of Theorem \ref{dimthree}.  

If $\bar{\mathcal{G}} (R)$ has $1$ block, then there is a variable which is contained in each layer of the graph.  Thus $\diam \bar {\mathcal{G}} (R) \leq \mu (d-1,n-1)$.  But by induction $\mu (d-1,n-1) \leq 2^{d-3} (n-d) \leq 2^{d-2}(n-d)$.

Now suppose we have multiple blocks.  Then the first block will be bound in diameter by $\mu (d-1, n_0-1)$, where $n_0$ is the number of variables in the block.  The second block will be bound in diameter by $\mu (d-1, n_0+n_1-d-1)$.  The third block will be bound by $\mu (d-1, n_1+n_2-d-1)$, and so on.  Using $k$ for  the number of blocks and using the induction hypothesis, we get: 

\[ \mu (d,n) \leq 2^{d-3}(n_0-d)+1+2^{d-3}(n_1+n_0-d-1-(d-1))+1+2^{d-3}(n_2+n_1-d-1-(d-1))+1 \]
\[ + \cdots + 2^{d-3}(n_{k-1}+n_{k-2}-d-1-(d-1)) \leq 2^{d-2}n -2^{d-3}(d+2d(k-1))+k-1 \leq \]
\[ 2^{d-2}n-2^{d-3}(3d)+1 \leq 2^{d-2}(n-d). \]

\end{proof}

\begin{remark}\label{upperboundremark}
In~\cite{EH10}, Eisenbrand et al. proved Larman's~\cite{La70} bound of $2^{d-1}n$ holds for graphs with property (i).  Theorem \ref{upperboundallnd} shows a stronger bound holds for graphs with properties (i) and (ii).  In~\cite{La70}, Larman showed that $2^{d-3}n$ is an upper bound for the diameter of polytopes of dimension at least 3.  In~\cite{Ba74}, Barnette strengthened Larman's bound to $\frac{1}{3}2^{d-3}(n-d+\frac{5}{2})$.  Our bound is slightly weaker than the bounds of Barnette and Larman; however, in Section \ref{glueconstruct} we will show by construction that the bounds of Barnette and Larman do not hold in our generality (see Theorem \ref{gluingd4} and \ref{gluingd3}).
\end{remark}

In~\cite{KW67}, Klee and Walkup prove $\mu (d,n) \leq \mu (n-d,2(n-d))$.  This fact gives rise to \textit{the $d$-step conjecture}, which states $\mu (d,2d) \leq d$ for all $d$.  We may rewrite this conjecture as $\mu (n-d,2(n-d)) \leq n-d$.  Thus, the d-step conjecture is equivalent to the Hirsch conjecture.  A natural generalization of the d-step conjecture is $\mu (d,2d) \leq p(d)$ where $p(d)$ is a polynomial in $d$.  Again $d=n-d$, and thus we have that this conjecture is equivalent to the polynomial Hirsch conjecture. 

We examine upper bounds on $\mu (d,d+k)$.  We note $\mu (d,d+k) = \mu(k,2k)$ by \cite{KW67}.

\begin{theorem}
$\mu (d, d+5) \leq 8$.
\end{theorem}

\begin{proof}

Choose any $v_1, v_2 \in \bar {\mathcal{G}} (R)$.  From~\cite{KW67}, we have $\mu (d,n) \leq \mu (n-d,2n-2d)$.  Thus we may reduce to the $d=5$ case.  Let $v_1 = ABCDE$.  We will consider cases based on $v_2$ to deduce the bound on diameter.

If $v_2 = FGHIJ$, then without loss of generality, we will have $BCDEF$ and $CDEFG$ in $V(\bar {\mathcal{G}} (R))$.  The vertices $CDEFG$ and $FGHIJ$ will have distance bounded by $\mu (3,8)$ (see Table \ref{table:1}).  Thus $\dist(CDEFG, FGHIJ) \leq 6$.  Thus $\dist ( v_1, v_2) \leq 8$.

If $v_2 = EFGHI$, then either $\bar{\mathcal{G}} (R)$ contains $ABCEF$ or $\bar{\mathcal{G}} (R)$ contains $ABCEJ$ and $ABEFJ$. For both scenarios, the graph will be bounded by two more than the $d=3$, $n=8$ case.  Thus $\dist (v_1, v_2) \leq 8$.

If $\degree (v_1 \cap v_2) \geq 2$, then $\dist (v_1, v_2)$ is bounded by the $d=3$, $n=8$ case and is at most $6$.

\end{proof}

\begin{theorem}
$\mu (d, d+6) \leq 14$.
\end{theorem}

\begin{proof}
Choose any $v_1, v_2 \in \bar {\mathcal{G}} (R)$.  As before, we can reduce to the $d=n-d$ case.  Let $v_1 = ABCDEF$.  

If $\degree (v_1 \cap v_2) \geq 3$, then $\diam \bar {\mathcal{G}} (R)$ is bounded by 7 (the $d=3$ $n=9$ case, see Table \ref{table:1}).

Now suppose $\degree (v_1 \cap v_2) \leq 2$.  There must exist a vertex $v_3$ such that $\degree (v_3 \cap v_1) \geq 3$ and $\degree (v_3 \cap v_2) \geq 3$.  

But then $\dist (v_1, v_3 ) \leq 7$, and $\dist (v_3, v_2 ) \leq 7$.  Thus $\dist (v_1, v_2) \leq 14$.

\end{proof}

\begin{theorem}\label{bigupperbound}
\[ \mu (d, d+k) \leq 3 \cdot 2^{\frac{n-d-5}{2}}(n-d). \]
\end{theorem}

\begin{proof}
We only need to consider the case $(n-d,2(n-d))$. 

We will first consider the case $n-d$ is even.  In this case, \[ \mu (n-d,2(n-d)) \leq 2 \mu \left ( \frac{1}{2} (n-d),\frac{3}{2} (n-d) \right ) \leq 2(2^{\frac{n-d}{2}-2}(n-d))=2^{\frac{n-d-2}{2}}(n-d) \leq 3 \cdot 2^{\frac{n-d-5}{2}}(n-d). \]

Next we consider the case $n-d$ is odd.  Then 

\[ \mu (n-d,2(n-d)) \leq \mu \left ( \left \lfloor{\frac{n-d}{2}}\right \rfloor , 2(n-d)- \left \lceil {\frac{n-d}{2}} \right \rceil \right ) + \mu \left ( \left \lceil{\frac{n-d}{2}}\right \rceil , 2(n-d)- \left \lfloor {\frac{n-d}{2}} \right \rfloor \right) \] 

\[ \leq  2^{\left \lfloor {\frac{n-d}{2}} \right \rfloor -2}(n-d)+ 2^{\left \lceil {\frac{n-d}{2}} \right \rceil -2}(n-d) = 3 \cdot 2^{\frac{n-d-5}{2}}(n-d). \]

\end{proof}

\section{Gluing}\label{glue}

In~\cite{Sa11}, Santos uses a gluing lemma to construct polyhedra with arbitrarily many facets whose diameters exceed $n-d$ by a fixed fraction.  We will construct an algebraic analogue to this gluing lemma, which will allow us to construct $(S_\ell)$ complexes with arbitrarily many facets whose diameters exceed $n-d$ by a fixed fraction.

The facet-ridge graph of an $(S_\ell)$ complex is $\bar{\mathcal{G}} (R),$ where $R$ is that complex's Stanley-Reisner ring.  Thus by making these $(S_2)$ complexes, we are making dual graphs of Stanley-Reisner rings satisfying $(S_2)$ with arbitrarily large $n$ whose diameters exceed $n-d$ by a fixed fraction.  

In this section, we will make use of the ext module $\Ext_S^i(R,S)$ and the cohomology module $H_{PR_P}^i(R_P)$.  For background on these modules, see a text on Homological Algebra (e.g. \cite{We94}).  We can describe $(S_\ell )$ in terms of the dimension of the Ext module.  This fact is key to the proof that proper gluing will maintain $(S_\ell )$.  The following lemma appears without proof in \cite[Proposition 3.1]{Va05}.

\begin{lemma}\label{gluelemma}
Let $\Delta $ be a pure complex with Stanley-Reisner ring $R$.  $R$ satisfies $(S_{\ell})$ if and only if \[ \dim \Ext _S ^{n-i} (R ,\omega _S) \leq i-{\ell} \quad \textrm{ for all } i < d. \]
\end{lemma}

\begin{proof}

We reconstruct the proof from \cite{DH16}.

Suppose $R$ satisfies $(S_l)$.  If $i < l \leq \depth R$ then $\Ext_S^{n-i}(R,S) = 0$.  Thus, \[ \dim \Ext_S^{n-i}(R,S) = -\infty, \] and we are finished.  Otherwise, let us take $P$ a prime ideal of $S$ containing $I$.  Let $h$ be the height of $P$.  If $h < n-i+l$, then $h-n+i<l$.  Further $\dim R_P = h-n+d > h-n+i$.  Thus $\depth R_P > h-n+i$.  Therefore,   
\[ 0= H_{PS_P}^{(h-n+i)}(R_P) \cong \Ext _{S_P}^{n-i}(R_P,S_P) \cong \Ext _S^{n-i}(R,S)_P \quad \textrm{ for all } i = 0,...,d-1. \]  
So $P \notin \Supp \Ext _S ^{n-i} (R,S)$ whenever $\height P < n-(i-l)$.   Thus $\dim \Ext _S^{n-i}(R,S) \leq i-l$.

Suppose $\dim \Ext _S^{n-i} (R,S) \leq i-l$ for all $i=0,...,d-1$.  Let $\height I = c$.  Let $V_h$ be the set of prime ideals of $S$ of height $h$ containing $I$ for $h=c,...,n$.  Then we have the following equivalences:

\[ R \textrm{ satisfies } (S_l) \quad \Leftrightarrow \]

\[ \depth R_P \geq \min \{ l, h-c \} =: b \quad \textrm{ for all } h = c,...,n, \textrm{ for all } P \in V_h \quad \Leftrightarrow \]

\[ H_{PS_P}^i(R_P)=0 \quad \textrm{ for all } h=c,...,n, \textrm{ for all } P \in V_h, \textrm{ for all } i < b \quad \Leftrightarrow \]

\[ \Ext _{S_P}^{h-i}(R_P,S_P)=0 \textrm{ for all } h=c,...,n, \textrm{ for all } P \in V_h, \textrm{ for all } i < b \quad \Leftrightarrow \]

\[ \dim \Ext _{S_P}^{h-i}(R_P,S_P)<n-h \textrm{ for all } h=c,...,n, \textrm{ for all } P \in V_h, \textrm{ for all } i < b. \]

When $i < b \leq h-c$, $n-h+i < n-c = d$.  For all $ i < d$, $\dim \Ext _S^{n-i}(R,S) \leq i-l$.  Thus 
\[ \dim \Ext _S ^{h-i} (R,S) = \dim \Ext _S ^{n-(n-h+i)}(R,S) \leq n-h+i-l < n-h \textrm{ for all } h = c,\cdots, n, \textrm{ for all } i < b. \]

\end{proof}

\begin{theorem}\label{gluing}
Let $\Delta$ and $\Delta '$ be $(d-1)$-dimensional complexes on $n$ vertices whose Stanley-Reisner rings each satisfy $(S_{\ell})$.  The Stanley-Reisner ring of $\Delta \cup \Delta '$ satisfies $(S_{\ell})$ if the two complexes are glued along a pure complex of dimension at least $d-2$ whose Stanley-Reisner ring satisfies $(S_{\ell -1})$.
\end {theorem}

\begin {proof}
Let us use the notation $R_{\Delta}$ to refer to the Stanley-Reisner ring of $\Delta$. 

If $\ell =1$, any gluing will preserve the $(S_1)$ property, since every simplicial complex satisfies $(S_1)$.

Thus let us consider $\ell \geq 2$, noting every $(S_2)$ complex is pure~\cite{Ya00}.

Take the short exact sequence

\[ 0 \rightarrow R_{\Delta \cup \Delta '} \rightarrow R_{\Delta} \oplus R_{\Delta '} \rightarrow R_{\Delta \cap \Delta '} \rightarrow 0. \]

Then take the long exact sequence in $\Ext$:
\[
\cdots \rightarrow \Ext ^{n-i}_S(R_\Delta \oplus R_{\Delta '}, \omega _S) \rightarrow \Ext ^{n-i}_S(R_{\Delta \cup \Delta '}, \omega _S) \rightarrow \Ext ^{n-(i-1)}_S(R_{\Delta \cap \Delta '}, \omega _S) \rightarrow \cdots.
\]

Since $\Ext $ is an additive functor: \[ \Ext ^{n-i}_S(R_\Delta \oplus R_{\Delta '}, \omega _S) \cong \Ext ^{n-i}_S(R_\Delta, \omega _S)  \oplus \Ext ^{n-i}_S(R_{\Delta '}, \omega _S). \]

Since $R_\Delta $ and $R_{\Delta '}$ are $(S_{\ell})$, Lemma \ref{gluelemma} gives \[ \dim \Ext ^{n-i}_S(R_{\Delta}, \omega _S) \leq i- \ell \] and \[ \dim \Ext ^{n-i}_S(R_{\Delta '}, \omega _S) \leq i-\ell. \] 

Therefore \[ \dim \Ext ^{n-i}_S(R_{\Delta}, \omega _S)  \oplus \Ext ^{n-i}_S(R_{\Delta '}, \omega _S) \leq i-\ell \quad \textrm{ for all } i < d. \]

Suppose $R_{\Delta \cap \Delta '}$ is an equidimensional, $(S_{\ell -1})$ ring of dimension at least $d-1$.  Then we also have \[ \dim \Ext ^{n-(i-1)}_S(R_{\Delta \cap \Delta '}, \omega _S) \leq (i-1)-(\ell -1) = i-\ell \quad \textrm{ for all } i < d. \] 

Therefore, \[ \dim \Ext ^{n-i}_S(R_{\Delta \cup \Delta '}, \omega _S) \leq i- \ell \quad \textrm{ for all } i < d. \]

\end{proof}

Thus, gluing any two $(S_2)$ complexes along a pure subcomplex of dimension at least $d-2$ yields an $(S_2)$ complex.  In particular, gluing two $(S_2)$ complexes along a facet yields an $(S_2)$ complex.

\section{Complexes Built by Gluing}\label{glueconstruct}

In this section, we create lower bounds for large $n$ and $d$ by taking copies of the graphs in Appendix \ref{AppendixA} and gluing them along a shared vertex.  The graphs from Appendix \ref{AppendixA} are facet-ridge graphs of $(S_2)$ complexes.  Thus, gluing along a vertex is equivalent to gluing the complexes along a facet.  Therefore, by Theorem \ref{gluing}, gluing along a vertex yields a facet-ridge graph of an $(S_2)$ complex.  

\begin{theorem}\label{gluingd4}
$\mu (4,4k+4) \geq 6k$. 
\end{theorem}

\begin{proof}
We construct a graph composed of $k$ copies of the graph in Figure \ref{dim4} by gluing the vertex $ABCD$ to the vertex $EFGH.$  Each copy adds $6$ to the diameter, since all adjacent vertices lie in the same copy of Figure \ref{dim4}.  The new graph retains local connectedness by Theorem \ref{gluing}.
\end{proof}

The complex in Figure \ref{4dfigure} is an example when $k=2$.  

Thus, we have a lower bound of $\frac{3}{2} (n-d)$ when $n=4k+4$ and $d=4$. 

\begin{cor}
$\mu (4,4k+4+j) \geq 6k+j$. 
\end{cor}

\begin{proof}
Start with the graph made of $k$ copies of Figure \ref{dim4}.  Then append the vertex 

$x_{6k-3}x_{6k-2}x_{6k-1}x_{6k+1}$.  If $j \geq 2$ then append  $x_{6k-2}x_{6k-1}x_{6k+1}x_{6k+2}$.  If $j = 3$ then append $x_{6k-1}x_{6k+1}x_{6k+2}x_{6k+3}$.
\end{proof}

For the $d=3$ case, we will consider three graphs:  The graph $G_0$ from Figure \ref{a4}, the graph $G_1$ from Figure \ref{a5}, and the graph $G_2$, which is the graph $G_1$ with $\{ I,J,K \}$ appended to the end.

\begin{theorem}\label{gluingd3}
$\mu (3,8k+2) \geq 10k-1$. 
\end{theorem}

\begin{proof}
Construct a graph composed of $k-1$ copies of $G_2$ by gluing $ABC$ to $IJK$.  Then glue a copy of $G_1$ to the end.  Each copy of $G_2$ adds $10$ to diameter, since all adjacent vertices lie in the same copy of $G_2$.  Gluing $G_1$ adds $9$ to the diameter.  The new graph retains local connectedness by Theorem \ref{gluing}.
\end{proof}

This yields a lower bound of $\frac{5}{4}(n-d)$ when $n=8k+2$, $d=3$.

\begin{cor}
$\mu (3,8k+3+j) \geq 10k+j$ when $j \geq 0$. 
\end{cor}

\begin{proof}
Take the graph constructed in Theorem \ref{gluingd3}.  Append the vertices $\{ n-2+i, n-1+i, n+i \}$ for $i = 1...j+1$.  This will be a locally connected graph of diameter $10k+j$.
\end{proof}

\begin{theorem}
$\mu (3,8k+3+j) \geq 10k+j+1$ when $j \geq 4$.
\end{theorem}

\begin{proof}
Glue $G_0$ to $k-1$ copies of $G_2$.  Then glue one copy of $G_1$.  If $j \geq 5$ then append $\{ n-6+i,n-5+i,n-4+i \}$ for $i = 5...j$.  This graph has diameter $10k+j+1$ and is locally connected by Theorem \ref{gluing}.
\end{proof}

Figure \ref{3dfigure} depicts the case where $d=3$, $n-d=12$.

\begin{remark}\label{gluingremark}
To construct a graph with $d \geq 4$, $\codim = 4k$, and diameter $\frac{3}{2}(4k)$, begin with the construction for $d=4$, $n=4k+4$ given above.  This construction has the desired diameter $\frac{3}{2}(4k)$.  Add the new variables $x_{n+1},...,x_{n+(d-4)}$ to each vertex of this graph to generate the desired graph.
\end{remark}

\begin{figure} [h]
\begin {center}
\resizebox {16cm} {!} {
\begin {tikzpicture}[-latex ,auto ,node distance =1.9 cm and 2.3cm ,on grid ,
semithick ,
state/.style ={ circle ,top color =white , bottom color = processblue!20 ,
draw,processblue , text=blue , inner sep=0pt, minimum width =.5 cm}]
\node[state] (C)
{$ABEG$};
\node[state] (A) [left=of C] {$BDEG$};
\node[state] (B) [right =of C] {$ACEG$};
\node[state] (D) [above =of C] {$ACEF$};
\node[state] (E) [below =of C] {$BDGH$};
\node[state] (F) [below =of E] {$CDFH$};
\node[state] (G) [right =of F] {$BDFH$};
\node[state] (H) [above =of G] {$CDGH$};
\node[state] (I) [left =of F] {$ACFH$};
\node[state] (K) [above =of A] {$CDEF$};
\node[state] (L) [below left =of K] {$BCDE$};
\node[state] (J) [below left =of L] {$ABCD$};
\node[state] (R) [left =of E] {$ABGH$};
\node[state] (M) [below right =of J] {$ABCH$};
\node[state] (N) [above =of B] {$ABEF$};
\node[state] (O) [right =of B] {$BEFH$};
\node[state] (P) [below right =of H] {$CEGH$};
\node[state] (Q) [above right =of P] {$EFGH$};
\node[state] (Y) [above right =of Q] {$FGHI$};
\node[state] (S) [below right =of Q] {$EFGL$};
\node[state] (T) [above right =of Y] {$GHIJ$};
\node[state] (U) [right =of Y] {$FGIK$};
\node[state] (V) [right =of T] {$EGIJ$};
\node[state] (W) [right =of V] {$EFIJ$};
\node[state] (X) [below right =of W] {$FIJL$};
\node[state] (Z) [right =of U] {$EFIK$};
\node[state] (AA) [right =of Z] {$EGIK$};
\node[state] (AB) [right =of S] {$EGJL$};
\node[state] (AC) [right =of AB] {$GHJL$};
\node[state] (AD) [right =of AC] {$FHJL$};
\node[state] (AE) [right =of AD] {$GIKL$};
\node[state] (AF) [above right =of AE] {$IJKL$};
\node[state] (AG) [above right =of S] {$EFKL$};
\node[state] (AH) [right =of AG] {$FHKL$};
\node[state] (AI) [right =of AH] {$GHKL$};

\path (W)[-] edge [bend left =0] node[below =0.15 cm] {$$} (X);
\path (AD)[-] edge [bend left =0] node[below =0.15 cm] {$$} (X);
\path (X)[-] edge [bend left =0] node[below =0.15 cm] {$$} (AF);
\path (AA)[-] edge [bend left =0] node[below =0.15 cm] {$$} (AE);
\path (AI)[-] edge [bend left =0] node[below =0.15 cm] {$$} (AE);
\path (AE)[-] edge [bend left =0] node[below =0.15 cm] {$$} (AF);
\path (AH)[-] edge [bend left =0] node[below =0.15 cm] {$$} (AI);
\path (AG)[-] edge [bend left =0] node[below =0.15 cm] {$$} (AH);
\path (S)[-] edge [bend left =0] node[below =0.15 cm] {$$} (AG);
\path (AG)[-] edge [bend left =0] node[below =0.15 cm] {$$} (Z);
\path (U)[-] edge [bend left =0] node[below =0.15 cm] {$$} (AH);
\path (AA)[-] edge [bend left =0] node[below =0.15 cm] {$$} (Z);
\path (Z)[-] edge [bend left =0] node[below =0.15 cm] {$$} (U);
\path (U)[-] edge [bend left =0] node[below =0.15 cm] {$$} (Y);
\path (AD)[-] edge [bend left =0] node[below =0.15 cm] {$$} (AC);
\path (AC)[-] edge [bend left =0] node[below =0.15 cm] {$$} (AB);
\path (AB)[-] edge [bend left =0] node[below =0.15 cm] {$$} (S);
\path (S)[-] edge [bend left =0] node[below =0.15 cm] {$$} (Q);
\path (Y)[-] edge [bend left =0] node[below =0.15 cm] {$$} (Q);
\path (T)[-] edge [bend left =0] node[below =0.15 cm] {$$} (Y);
\path (V)[-] edge [bend left =0] node[below =0.15 cm] {$$} (T);
\path (W)[-] edge [bend left =0] node[below =0.15 cm] {$$} (V);
\path (T)[-] edge [bend left =0] node[below =0.15 cm] {$$} (AC);
\path (V)[-] edge [bend left =0] node[below =0.15 cm] {$$} (AB);
\path (Z)[-] edge [bend left =0] node[below =0.15 cm] {$$} (W);
\path (AA)[-] edge [bend left =0] node[below =0.15 cm] {$$} (V);
\path (AC)[-] edge [bend left =0] node[below =0.15 cm] {$$} (AI);
\path (AH)[-] edge [bend left =0] node[below =0.15 cm] {$$} (AD);

\path (J)[-] edge [bend left =0] node[below =0.15 cm] {$$} (M);
\path (J)[-] edge [bend left =0] node[below =0.15 cm] {$$} (L);
\path (M)[-] edge [bend left =0] node[below =0.15 cm] {$$} (I);
\path (M)[-] edge [bend left =0] node[below =0.15 cm] {$$} (R);
\path (L)[-] edge [bend left =0] node[below =0.15 cm] {$$} (A);
\path (L)[-] edge [bend left =0] node[below =0.15 cm] {$$} (K);
\path (K)[-] edge [bend left =0] node[below =0.15 cm] {$$} (D);
\path (A)[-] edge [bend left =0] node[below =0.15 cm] {$$} (C);
\path (R)[-] edge [bend left =0] node[below =0.15 cm] {$$} (E);
\path (I)[-] edge [bend left =0] node[below =0.15 cm] {$$} (F);
\path (D)[-] edge [bend left =0] node[below =0.15 cm] {$$} (N);
\path (C)[-] edge [bend left =0] node[below =0.15 cm] {$$} (B);
\path (E)[-] edge [bend left =0] node[below =0.15 cm] {$$} (H);
\path (F)[-] edge [bend left =0] node[below =0.15 cm] {$$} (G);
\path (N)[-] edge [bend left =0] node[below =0.15 cm] {$$} (O);
\path (B)[-] edge [bend left =0] node[below =0.15 cm] {$$} (P);
\path (H)[-] edge [bend left =0] node[below =0.15 cm] {$$} (P);
\path (G)[-] edge [bend left =0] node[below =0.15 cm] {$$} (O);
\path (P)[-] edge [bend left =0] node[below =0.15 cm] {$$} (Q);
\path (O)[-] edge [bend left =0] node[below =0.15 cm] {$$} (Q);
\path (A)[-] edge [bend left =0] node[below =0.15 cm] {$$} (E);
\path (R)[-] edge [bend left =0] node[below =0.15 cm] {$$} (C);
\path (K)[-] edge [bend left =0] node[below =0.15 cm] {$$} (F);
\path (I)[-] edge [bend left =0] node[below =0.15 cm] {$$} (D);
\path (D)[-] edge [bend left =0] node[below =0.15 cm] {$$} (B);
\path (C)[-] edge [bend left =0] node[below =0.15 cm] {$$} (N);
\path (E)[-] edge [bend left =0] node[below =0.15 cm] {$$} (G);
\path (F)[-] edge [bend left =0] node[below =0.15 cm] {$$} (H);
\end{tikzpicture}
}
\end{center}
\caption{}
\label{4dfigure}
\end{figure}

\begin{figure}[h]

\begin {center}
\resizebox {16cm} {!} {
\begin {tikzpicture}[-latex ,auto ,node distance =1.6 cm and 1.6cm ,on grid ,
semithick ,
state/.style ={ circle ,top color =white , bottom color = processblue!20 ,
draw,processblue , text=blue , inner sep=0pt, minimum width =.4 cm}]
\node[state] (C)
{$DJN$};
\node[state] (A) [above=of C] {$EJO$};
\node[state] (B) [below =of C] {$FJM$};
\node[state] (D) [left =of C] {$DJM$};
\node[state] (E) [left =of B] {$FJO$};
\node[state] (F) [left =of A] {$EJN$};
\node[state] (G) [left =of F] {$EIN$};
\node[state] (H) [left =of E] {$FIO$};
\node[state] (I) [left =of D] {$DIM$};
\node[state] (J) [left =of G] {$EFI$};
\node[state] (K) [left =of I] {$DFI$};
\node[state] (L) [above left =of K] {$DEF$};
\node[state] (M) [right =of C] {$DKN$};
\node[state] (N) [right =of B] {$FKM$};
\node[state] (O) [right =of A] {$EKO$};
\node[state] (P) [right =of O] {$EKM$};
\node[state] (Q) [right =of N] {$FKN$};
\node[state] (R) [right =of M] {$DKO$};
\node[state] (S) [right =of P] {$ELM$};
\node[state] (T) [right =of Q] {$FLN$};
\node[state] (U) [right =of R] {$DLO$};
\node[state] (V) [right =of S] {$LMO$};
\node[state] (W) [right =of U] {$LMN$};
\node[state] (X) [above right =of W] {$MNO$};
\node[state] (Z) [left =of L] {$DEH$};
\node[state] (AA) [left =of Z] {$AEH$};
\node[state] (AB) [left =of AA] {$AEG$};
\node[state] (AC) [left =of AB] {$ADG$};
\node[state] (AD) [left =of AC] {$ABD$};
\node[state] (AE) [left =of AD] {$ABC$};
\node[state] (AF) [below =of AA] {$CDH$};
\node[state] (AG) [left =of AF] {$CDG$};
\node[state] (AH) [left =of AG] {$CEG$};
\node[state] (AI) [left =of AH] {$BCE$};

\path [-] (L) edge [bend left =0] node[below =0.15 cm] {$$} (K);
\path[-] (J) edge [bend left =0] node[below =0.15 cm] {$$} (L);
\path[-] (J) edge [bend left =0] node[below =0.15 cm] {$$} (K);
\path[-] (J) edge [bend left =0] node[below =0.15 cm] {$$} (G);
\path[-] (K) edge [bend left =0] node[below =0.15 cm] {$$} (H);
\path[-] (J) edge [bend left =0] node[below =0.15 cm] {$$} (H);
\path[-] (K) edge [bend left =0] node[below =0.15 cm] {$$} (I);
\path[-] (G) edge [bend left =0] node[below =0.15 cm] {$$} (F);
\path[-] (H) edge [bend left =0] node[below =0.15 cm] {$$} (E);
\path[-] (I) edge [bend left =0] node[below =0.15 cm] {$$} (D);
\path [-] (D) edge [bend left =0] node[below =0.15 cm] {$$} (C);
\path[-] (E) edge [bend left =0] node[below =0.15 cm] {$$} (B);
\path[-] (F) edge [bend left =0] node[below =0.15 cm] {$$} (A);
\path[-] (A) edge [bend left =0] node[below =0.15 cm] {$$} (O);
\path[-] (C) edge [bend left =0] node[below =0.15 cm] {$$} (M);
\path[-] (B) edge [bend left =0] node[below =0.15 cm] {$$} (N);
\path[-] (D) edge [bend left =0] node[below =0.15 cm] {$$} (B);
\path[-] (E) edge [bend left =0] node[below =0.15 cm] {$$} (A);
\path[-] (F) edge [bend left =0] node[below =0.15 cm] {$$} (C);
\path[-] (P) edge [bend left =0] node[below =0.15 cm] {$$} (O);
\path[-] (R) edge [bend left =0] node[below =0.15 cm] {$$} (M);
\path[-] (Q) edge [bend left =0] node[below =0.15 cm] {$$} (N);
\path[-] (Q) edge [bend left =0] node[below =0.15 cm] {$$} (M);
\path[-] (P) edge [bend left =0] node[below =0.15 cm] {$$} (N);
\path[-] (R) edge [bend left =0] node[below =0.15 cm] {$$} (O);
\path[-] (P) edge [bend left =0] node[below =0.15 cm] {$$} (S);
\path[-] (R) edge [bend left =0] node[below =0.15 cm] {$$} (U);
\path[-] (Q) edge [bend left =0] node[below =0.15 cm] {$$} (T);
\path[-] (S) edge [bend left =0] node[below =0.15 cm] {$$} (W);
\path[-] (U) edge [bend left =0] node[below =0.15 cm] {$$} (V);
\path[-] (T) edge [bend left =0] node[below =0.15 cm] {$$} (W);
\path[-] (S) edge [bend left =0] node[below =0.15 cm] {$$} (V);
\path[-] (X) edge [bend left =0] node[below =0.15 cm] {$$} (W);
\path[-] (X) edge [bend left =0] node[below =0.15 cm] {$$} (V);
\path[-] (V) edge [bend left =0] node[below =0.15 cm] {$$} (W);

\path [-] (L) edge [bend left =0] node[below =0.15 cm] {$$} (Z);
\path[-] (Z) edge [bend left =0] node[below =0.15 cm] {$$} (AA);
\path[-] (AA) edge [bend left =0] node[below =0.15 cm] {$$} (AB);
\path[-] (AB) edge [bend left =0] node[below =0.15 cm] {$$} (AC);
\path [-] (AC) edge [bend left =0] node[below =0.15 cm] {$$} (AD);
\path[-] (AD) edge [bend left =0] node[below =0.15 cm] {$$} (AE);
\path[-] (AF) edge [bend left =0] node[below =0.15 cm] {$$} (Z);
\path[-] (AF) edge [bend left =0] node[below =0.15 cm] {$$} (AG);
\path[-] (AG) edge [bend left =0] node[below =0.15 cm] {$$} (AH);
\path[-] (AH) edge [bend left =0] node[below =0.15 cm] {$$} (AI);
\path[-] (AE) edge [bend left =0] node[below =0.15 cm] {$$} (AI);
\path[-] (AB) edge [bend left =0] node[below =0.15 cm] {$$} (AH);
\path[-] (AC) edge [bend left =0] node[below =0.15 cm] {$$} (AG);
\end{tikzpicture}
}
\caption{}
\label{3dfigure}
\end{center}
\end{figure}

\section{Final Remarks}

These results raise several questions for further study.  Primarily, is $\mu (d,n)$ bounded above by a polynomial in the codimension?  A positive answer to this question would affirm the polynomial Hirsch conjecture.  Our work still leaves the possibility that $\mu (d,n)$ is bounded by a linear function.

Another interesting question is, do the bounds of Larman~\cite{La70} and Barnette~\cite{Ba74} hold for larger values of $d$?  We have seen that the bounds of Larman and Barnette do not hold for our $d=3$ case.  Figure \ref{4dfigure} shows the bound of Barnette does not hold for $d=4$, and we can cone over this complex to show the bound of Barnette does not hold for $d=5$; however, we do not have counterexamples in any higher dimensions.

In this paper we found graphs of maximal diameter for $(S_2)$ Stanley-Reisner rings with $d=3,4$ and small $n$.  We then used those graphs in conjunction with gluing to make graphs with large diameters with respect to codimension.  It would be valuable to know what the largest diameter would be for graphs of $(S_2)$ Stanley-Reisner rings with $d=5,6$ and small $n$, specifically, $d=5,n=10$ and $d=6,n=12$.  Answers to these questions could lead to new asymptotic lower bounds and could give insight on how these bounds would grow with respect to $d$.

\appendix\section{Constructing graphs to identify lower bounds}\label{AppendixA}

In this appendix, for $d$ and $n$ fixed, we construct lower bounds for the maximum diameters of dual graphs of equidimensional Stanley-Reisner rings satisfying $(S_2)$.  We achieve this by constructing graphs with properties (i) and (ii).  Recall, $\mu (d,n)$ is the largest diameter of a dual graph of an equidimensional, $(S_2)$ Stanley-Reisner ring of dimension $d$ and codimension $n-d$.  

\begin{theorem}
Table 1 (see Introduction) presents $\mu (d,n)$ for small values of $d$ and $n$. 
\end{theorem}

This theorem is proved by the propositions of this appendix.

\begin{prop}\label{AP2}
$\mu (2,n) = n-2$.
\end{prop}

\begin{proof}
$R$ satisfies $(S_2)$ if and only if $\bar{\mathcal{G}} (R)$ is locally connected.  In the $d=2$ case, locally connected is equivalent to connected.  Thus we construct a connected graph.  Create a vertex $v_1 = x_1x_2$.  We wish to create another vertex $v_2$ adjacent to $v_1$.  Without loss of generality, $v_2 = x_1x_3$.  Any vertex not adjacent to $v_1$ but adjacent to $v_2$ must be of the form $x_3x_i$ $(i \neq 1,2,3)$.  Thus we may choose $v_3=x_3x_4$.  Continuing this process, we see that $\mu (2,n)$ is bounded above by the number of variables in $R$ that are not contained in $v_1$.  We also see that this construction yields a graph of diameter $n-2$.  Thus $\mu (2,n) = n-2$.  Figure \ref{a1} shows a graph with properties (i) and (ii) of diameter $3$ when $n=5, d=2$.
\end{proof}

\begin{figure}[h]
\begin {center}
\resizebox {8cm} {!} {
\begin {tikzpicture}[-latex ,auto ,node distance =1.35 cm and 1.75cm ,on grid ,
semithick ,
state/.style ={ circle ,top color =white , bottom color = processblue!20 ,
draw,processblue , text=blue , inner sep=0pt, minimum width = .8 cm}]
\node[state] (C)
{$BC$};
\node[state] (A) [left=of C] {$AB$};
\node[state] (B) [right =of C] {$CD$};
\node[state] (D) [right =of B] {$DE$};
\path[-] (A) edge [bend left =0] node[below =0.15 cm] {$$} (C);
\path[-] (C) edge [bend left =0] node[below =0.15 cm] {$$} (B);
\path[-] (B) edge [bend left =0] node[below =0.15 cm] {$$} (D);
\end{tikzpicture}
}

\caption{}
\label{a1}
\end{center}
\end{figure}

\begin{prop}
$\mu (3,6) =3$.
\end{prop}

\begin{proof}
Let us consider any distinct pair of vertices $v_1,v_2 \in \bar {\mathcal{G}} (R)$.  If $\deg (v_1 \cap v_2) = 2$ then $v_1,v_2$ are adjacent.  If $\deg (v_1 \cap v_2) = 1$, then local connectedness of the graph requires that there exists a path from $v_1$ to $v_2$ such that each vertex in the path contains $v_1 \cap v_2$.  Thus applying Proposition \ref{AP2}, the distance between these two vertices is bounded above by $\mu (2,5) = 3$.  If $\deg (v_1 \cap v_2) = 0$, then every vertex is either adjacent to $v_1$ or adjacent to $v_2$.  Therefore $\mu (3,6) \leq 3$.  Adding $F$ to every vertex label in Figure \ref{a1} produces a diameter-3 graph with properties (i) and (ii) with $d=3$, $n=6$.  Therefore $\mu (3,6) =3$.
\end{proof}

\begin{cor}
$\mu (d,d+3) = 3$.
\end{cor}

\begin{proof}
Let $R$ be a codimension-$3$, $(S_2)$ Stanley-Reisner ring.  Take $v_1, v_2 \in V(\bar {\mathcal{G}} (R))$.  Then, $v_1, v_2$ will each contain $n-3$ variables, and $v_1 \cap v_2$ will contain at least $n-6$ variables.  Thus there must be a path from $v_1$ to $v_2$ in which each vertex in that path contains those $n-6$ shared variables.  Thus $\mu (d,d+3) \leq \mu (3,6)$.  Furthermore, we may take the graph with vertex set $\{ABC, BCD, CDE, DEF \}$ and add the same $d-3$ variables to each vertex to show $\mu (d,d+3) \geq 3$.
\end{proof}

\begin{prop}
$\mu (3,7) =5$.
\end{prop}

\begin{proof}
Figure \ref{a2} is an example of a diameter-$5$ graph with properties (i) and (ii) with $d=3$, $n=7$.  Thus $\mu (3,7) \geq 5$.  In Theorem \ref{dimthree}, we proved $\mu (3,7) \leq 5$.
\end{proof}

\begin{figure} [h]
\begin {center}
\begin {tikzpicture}[-latex ,auto ,node distance =1.75 cm and 1.55cm ,on grid ,
semithick ,
state/.style ={ circle ,top color =white , bottom color = processblue!20 ,
draw,processblue , text=blue , inner sep=0pt, minimum width =.4 cm}]
\node[state] (C)
{$CDG$};
\node[state] (A) [above=of C] {$AEG$};
\node[state] (D) [left =of C] {$CEG$};
\node[state] (F) [left =of A] {$ADG$};
\node[state] (G) [left =of F] {$ABD$};
\node[state] (I) [left =of D] {$BCE$};
\node[state] (J) [left =of G] {$ABC$};
\node[state] (K) [right =of A] {$AEF$};
\node[state] (L) [right =of C] {$CDF$};
\node[state] (M) [right =of K] {$DEF$};

\path [-] (D) edge [bend left =0] node[below =0.15 cm] {$$} (C);
\path[-] (F) edge [bend left =0] node[below =0.15 cm] {$$} (A);
\path[-] (F) edge [bend left =0] node[below =0.15 cm] {$$} (C);
\path[-] (K) edge [bend left =0] node[below =0.15 cm] {$$} (A);
\path [-] (L) edge [bend left =0] node[below =0.15 cm] {$$} (M);
\path[-] (K) edge [bend left =0] node[below =0.15 cm] {$$} (M);
\path[-] (C) edge [bend left =0] node[below =0.15 cm] {$$} (L);
\path[-] (I) edge [bend left =0] node[below =0.15 cm] {$$} (D);
\path[-] (A) edge [bend left =0] node[below =0.15 cm] {$$} (D);
\path[-] (G) edge [bend left =0] node[below =0.15 cm] {$$} (J);
\path[-] (I) edge [bend left =0] node[below =0.15 cm] {$$} (J);
\path[-] (G) edge [bend left =0] node[below =0.15 cm] {$$} (F);

\end{tikzpicture}

\caption{}
\label{a2}
\end{center}
\end{figure}

To see that $\bar{\mathcal{G}} (R)$ is locally connected, we take any subset of variables $S$ and check that the vertices containing $S$ form a connected subgraph.  Below, we color the variables containing $E$ red.  
\begin{figure}[h]
\begin {center}
\begin {tikzpicture}[-latex ,auto ,node distance =1.75 cm and 1.55cm ,on grid ,
semithick ,
state/.style ={ circle ,top color =white , bottom color = processblue!20 ,
draw,processblue , text=blue , inner sep=0pt, minimum width =.4 cm}]
\node[state] (C)
{$CDG$};
\node[{ circle ,top color =red , bottom color = red!40 ,
draw,processblue , text=blue , inner sep=0pt, minimum width =.4 cm}] (A) [above=of C] {$AEG$};
\node[{ circle ,top color =red , bottom color = red!40 ,
draw,processblue , text=blue , inner sep=0pt, minimum width =.4 cm}] (D) [left =of C] {$CEG$};
\node[state] (F) [left =of A] {$ADG$};
\node[state] (G) [left =of F] {$ABD$};
\node[{ circle ,top color =red , bottom color = red!40 ,
draw,processblue , text=blue , inner sep=0pt, minimum width =.4 cm}] (I) [left =of D] {$BCE$};
\node[state] (J) [left =of G] {$ABC$};
\node[{ circle ,top color =red , bottom color = red!40 ,
draw,processblue , text=blue , inner sep=0pt, minimum width =.4 cm}] (K) [right =of A] {$AEF$};
\node[state] (L) [right =of C] {$CDF$};
\node[{ circle ,top color =red , bottom color = red!40 ,
draw,processblue , text=blue , inner sep=0pt, minimum width =.4 cm}] (M) [right =of K] {$DEF$};

\path [-] (D) edge [bend left =0] node[below =0.15 cm] {$$} (C);
\path[-] (F) edge [bend left =0] node[below =0.15 cm] {$$} (A);
\path[-] (F) edge [bend left =0] node[below =0.15 cm] {$$} (C);
\path[-] (K) edge [bend left =0] node[below =0.15 cm] {$$} (A);
\path [-] (L) edge [bend left =0] node[below =0.15 cm] {$$} (M);
\path[-] (K) edge [bend left =0] node[below =0.15 cm] {$$} (M);
\path[-] (C) edge [bend left =0] node[below =0.15 cm] {$$} (L);
\path[-] (I) edge [bend left =0] node[below =0.15 cm] {$$} (D);
\path[-] (A) edge [bend left =0] node[below =0.15 cm] {$$} (D);
\path[-] (G) edge [bend left =0] node[below =0.15 cm] {$$} (J);
\path[-] (I) edge [bend left =0] node[below =0.15 cm] {$$} (J);
\path[-] (G) edge [bend left =0] node[below =0.15 cm] {$$} (F);

\end{tikzpicture}

\caption{}
\label{a3}
\end{center}
\end{figure}

\begin{prop}
$\mu (3,8) =6$.
\end{prop}

\begin{proof}
Figure \ref{a4} is an example of a diameter-$6$ graph with properties (i) and (ii) with $d=3$, $n=8$.  Thus $\mu (3,8) \geq 6$.  In Theorem \ref{dimthree}, we proved $\mu (3,8) \leq 6$.
\end{proof}


\begin{figure}[h]
\begin {center}
\begin {tikzpicture}[-latex ,auto ,node distance =1.75 cm and 1.55cm ,on grid ,
semithick ,
state/.style ={ circle ,top color =white , bottom color = processblue!20 ,
draw,processblue , text=blue , inner sep=0pt, minimum width =.4 cm}]
\node[state] (C)
{$CDG$};
\node[state] (A) [above=of C] {$AEG$};
\node[state] (D) [left =of C] {$CEG$};
\node[state] (F) [left =of A] {$ADG$};
\node[state] (G) [left =of F] {$ABD$};
\node[state] (I) [left =of D] {$BCE$};
\node[state] (J) [left =of G] {$ABC$};
\node[state] (K) [right =of A] {$AEF$};
\node[state] (L) [right =of C] {$CDF$};
\node[state] (M) [right =of K] {$DEF$};
\node[state] (N) [right =of M] {$DEH$};

\path [-] (D) edge [bend left =0] node[below =0.15 cm] {$$} (C);
\path[-] (F) edge [bend left =0] node[below =0.15 cm] {$$} (A);
\path[-] (F) edge [bend left =0] node[below =0.15 cm] {$$} (C);
\path[-] (K) edge [bend left =0] node[below =0.15 cm] {$$} (A);
\path [-] (L) edge [bend left =0] node[below =0.15 cm] {$$} (M);
\path[-] (K) edge [bend left =0] node[below =0.15 cm] {$$} (M);
\path[-] (C) edge [bend left =0] node[below =0.15 cm] {$$} (L);
\path[-] (I) edge [bend left =0] node[below =0.15 cm] {$$} (D);
\path[-] (A) edge [bend left =0] node[below =0.15 cm] {$$} (D);
\path[-] (G) edge [bend left =0] node[below =0.15 cm] {$$} (J);
\path[-] (I) edge [bend left =0] node[below =0.15 cm] {$$} (J);
\path[-] (G) edge [bend left =0] node[below =0.15 cm] {$$} (F);
\path[-] (M) edge [bend left =0] node[below =0.15 cm] {$$} (N);

\end{tikzpicture}

\caption{}
\label{a4}
\end{center}
\end{figure}

\begin{prop}\label{AP6}
$\mu (3,9) =7$.
\end{prop}

\begin{proof}
We can prove that there does not exist $\bar{\mathcal{G}} (R)$, such that $R$ is an $(S_2)$ Stanley-Reisner ring, $n=9$, $d=3$, and $\diam(\bar {\mathcal{G}} (R)) > 7$.  Due to length, this proof has been omitted from the paper but can be found at: \url{http://www.math.ku.edu/~b101h187/}.   The graph in Figure \ref{a4} with the vertex $EHI$ appended is a diameter-7 graph with properties (i) and (ii) when $n=9,d=3$.
\end{proof}

Proposition \ref{AP6} gives a bound only one better than the general upper bound given in Theorem \ref{dimthree}.

\begin{prop}\label{AP8}
$\mu (3,10) \geq 9$.
\end{prop}

\begin{proof}
Figure \ref{a5} is an example of a diameter-$9$ graph with properties (i) and (ii) with $d=3$, $n=10$.
\end{proof}

Proposition \ref{AP8} gives a bound only one better than the bound of Theorem \ref{dimthree}.

\begin{figure} [h]
\begin {center}
\begin {tikzpicture}[-latex ,auto ,node distance =1.75 cm and 1.55cm ,on grid ,
semithick ,
state/.style ={ circle ,top color =white , bottom color = processblue!20 ,
draw,processblue , text=blue , inner sep=0pt, minimum width =.4 cm}]
\node[state] (C)
{$AEI$};
\node[state] (A) [above=of C] {$BEJ$};
\node[state] (B) [below =of C] {$CEH$};
\node[state] (D) [left =of C] {$AEH$};
\node[state] (E) [left =of B] {$CEJ$};
\node[state] (F) [left =of A] {$BEI$};
\node[state] (G) [left =of F] {$BDI$};
\node[state] (H) [left =of E] {$CDJ$};
\node[state] (I) [left =of D] {$ADH$};
\node[state] (J) [left =of G] {$BCD$};
\node[state] (K) [left =of I] {$ACD$};
\node[state] (L) [above left =of K] {$ABC$};
\node[state] (M) [right =of C] {$AFI$};
\node[state] (N) [right =of B] {$CFH$};
\node[state] (O) [right =of A] {$BFJ$};
\node[state] (P) [right =of O] {$BFH$};
\node[state] (Q) [right =of N] {$CFI$};
\node[state] (R) [right =of M] {$AFJ$};
\node[state] (S) [right =of P] {$BGH$};
\node[state] (T) [right =of Q] {$CGI$};
\node[state] (U) [right =of R] {$AGJ$};
\node[state] (V) [right =of S] {$GHJ$};
\node[state] (W) [right =of U] {$GHI$};
\node[state] (X) [above right =of W] {$HIJ$};

\path [-] (J) edge [bend left =0] node[below =0.15 cm] {$$} (K);
\path [-] (L) edge [bend left =0] node[below =0.15 cm] {$$} (K);
\path[-] (J) edge [bend left =0] node[below =0.15 cm] {$$} (L);
\path[-] (J) edge [bend left =0] node[below =0.15 cm] {$$} (G);
\path[-] (K) edge [bend left =0] node[below =0.15 cm] {$$} (H);
\path[-] (J) edge [bend left =0] node[below =0.15 cm] {$$} (H);
\path[-] (K) edge [bend left =0] node[below =0.15 cm] {$$} (I);
\path[-] (G) edge [bend left =0] node[below =0.15 cm] {$$} (F);
\path[-] (H) edge [bend left =0] node[below =0.15 cm] {$$} (E);
\path[-] (I) edge [bend left =0] node[below =0.15 cm] {$$} (D);
\path [-] (D) edge [bend left =0] node[below =0.15 cm] {$$} (C);
\path[-] (E) edge [bend left =0] node[below =0.15 cm] {$$} (B);
\path[-] (F) edge [bend left =0] node[below =0.15 cm] {$$} (A);
\path[-] (A) edge [bend left =0] node[below =0.15 cm] {$$} (O);
\path[-] (C) edge [bend left =0] node[below =0.15 cm] {$$} (M);
\path[-] (B) edge [bend left =0] node[below =0.15 cm] {$$} (N);
\path[-] (D) edge [bend left =0] node[below =0.15 cm] {$$} (B);
\path[-] (E) edge [bend left =0] node[below =0.15 cm] {$$} (A);
\path[-] (F) edge [bend left =0] node[below =0.15 cm] {$$} (C);
\path[-] (P) edge [bend left =0] node[below =0.15 cm] {$$} (O);
\path[-] (R) edge [bend left =0] node[below =0.15 cm] {$$} (M);
\path[-] (Q) edge [bend left =0] node[below =0.15 cm] {$$} (N);
\path[-] (Q) edge [bend left =0] node[below =0.15 cm] {$$} (M);
\path[-] (P) edge [bend left =0] node[below =0.15 cm] {$$} (N);
\path[-] (R) edge [bend left =0] node[below =0.15 cm] {$$} (O);
\path[-] (P) edge [bend left =0] node[below =0.15 cm] {$$} (S);
\path[-] (R) edge [bend left =0] node[below =0.15 cm] {$$} (U);
\path[-] (Q) edge [bend left =0] node[below =0.15 cm] {$$} (T);
\path[-] (S) edge [bend left =0] node[below =0.15 cm] {$$} (W);
\path[-] (U) edge [bend left =0] node[below =0.15 cm] {$$} (V);
\path[-] (T) edge [bend left =0] node[below =0.15 cm] {$$} (W);
\path[-] (S) edge [bend left =0] node[below =0.15 cm] {$$} (V);
\path[-] (X) edge [bend left =0] node[below =0.15 cm] {$$} (W);
\path[-] (X) edge [bend left =0] node[below =0.15 cm] {$$} (V);
\path[-] (W) edge [bend left =0] node[below =0.15 cm] {$$} (V);

\end{tikzpicture}

\caption{}
\label{a5}
\end{center}
\end{figure}

\begin{prop}
$\mu (3,n) \geq n-d+2$ for all $n \geq 10$.
\end{prop}

\begin{proof}
We can construct an example of a diameter-$(n-d+2)$ graph with properties (i) and (ii) with $d=3$, $n=10+j$ by taking the graph in Figure \ref{a5} and appending the vertices: \[ IJx_{1}, Jx_1x_2, x_1x_2x_3, \cdots ,x_{j-2}x_{j-1}x_{j}. \]
\end{proof}

Buchsbaum complexes have long been studied in combinatorial algebra \cite{HT96, Te96, HK01, TY06}.  It is of interest that all of the complexes we have examined thus far are Buchsbaum.  The following is likely known to experts, but we include a proof here.

\begin{prop}
Let $R=S/I$ be an equidimensional Stanley-Reisner ring of dimension $3$.  Then $R$ is connected and Buchsbaum if and only if $R$ satisfies $(S_2)$.
\end{prop}

\begin{proof}
Let $\Delta$ be a complex with Stanley-Reisner ring $R = S/I$.  Let $m$ be the unique homogeneous maximal ideal of $S$.  Using local duality (see Lemma \ref{gluelemma}), we have that a $3$-dimensional equidimensional Stanley-Reisner ring $R$ satisfies $(S_2)$ if and only if \[ \dim \Ext _S ^{n-i} (R,\omega _S) \leq i-2 \quad \textrm{ for all } i < 3. \]  Thus, $R$ satisfies $(S_2)$ if and only if $H_m^0(R) = H_m^1(R)=0$ and $H_m^2(R)$ is finitely generated.  

Suppose $R$ is Buchsbaum and connected.  Connectivity implies $H_m^0(R) = H_m^1(R)=0$.  If $R$ is Buchsbaum, then $R$ is Generalized Cohen-Macaulay, which implies that $H_m^i(R)$ is finitely generated for all $i < d$.  Thus for $R$ an equidimensional Stanley-Reisner ring of dimension $3$, Buchsbaum and connected imply $(S_2)$. 

Suppose now that $R$ satisfies $(S_2)$.  We will use the combinatorial definition of Buchsbaum, which says a complex is Buchsbaum if it is pure and has the property that the link of any non-empty face has zero reduced homology except possibly in top dimension.  

We first note the Stanley-Reisner ring of $\Delta$ satisfying $(S_2)$ implies $\Delta$ is pure and connected.  Next we note that every link of a non-empty face of $\Delta$ has Stanley-Reisner ring $R_P$, where $P$ is a prime ideal and $\dim R_P < \dim R$.  Thus $R$ being a $3$-dimensional, $(S_2)$ ring implies $R_P$ is Cohen-Macaulay for all $P$ such that $\dim R_P < \dim R$.  Thus $R_P$ has zero reduced homology except possibly in top dimension.  Thus $R$ is Buchsbaum and connected.
\end{proof}

Note that this theorem does not apply in the higher dimension cases.  In fact, most of our examples in higher dimension, including Figure \ref{dim4} below, are not Buchsbaum.

\begin{prop}
$\mu (4,8) = 6$.
\end{prop}

\begin{proof}
Figure \ref{dim4} is an example of a diameter-$6$ graph with propeties (i) and (ii) with $d=4$, $n=8$.  In Corollary \ref{C3}, we will prove $\mu (4,8) \leq 6$.
\end{proof}

\begin{figure}[h]
\begin {center}
\resizebox {9cm} {!} {
\begin {tikzpicture}[-latex ,auto ,node distance =2.07 cm and 2.3cm ,on grid ,
semithick ,
state/.style ={ circle ,top color =white , bottom color = processblue!20 ,
draw,processblue , text=blue , inner sep=0pt, minimum width =.3 cm}]
\node[state] (C)
{$ABEG$};
\node[state] (A) [left=of C] {$BDEG$};
\node[state] (B) [right =of C] {$ACEG$};
\node[state] (D) [above =of C] {$ACEF$};
\node[state] (E) [below =of C] {$BDGH$};
\node[state] (F) [below =of E] {$CDFH$};
\node[state] (G) [right =of F] {$BDFH$};
\node[state] (H) [above =of G] {$CDGH$};
\node[state] (I) [left =of F] {$ACFH$};
\node[state] (K) [above =of A] {$CDEF$};
\node[state] (L) [below left =of K] {$BCDE$};
\node[state] (J) [below left =of L] {$ABCD$};
\node[state] (R) [left =of E] {$ABGH$};
\node[state] (M) [below right =of J] {$ABCH$};
\node[state] (N) [above =of B] {$ABEF$};
\node[state] (O) [above right =of B] {$BEFH$};
\node[state] (P) [right =of H] {$CEGH$};
\node[state] (Q) [below right =of O] {$EFGH$};
\path[-] (J) edge [bend left =0] node[below =0.15 cm] {$$} (M);
\path[-] (J) edge [bend left =0] node[below =0.15 cm] {$$} (L);
\path[-] (M) edge [bend left =0] node[below =0.15 cm] {$$} (I);
\path[-] (M) edge [bend left =0] node[below =0.15 cm] {$$} (R);
\path[-] (L) edge [bend left =0] node[below =0.15 cm] {$$} (A);
\path[-] (L) edge [bend left =0] node[below =0.15 cm] {$$} (K);
\path[-] (K) edge [bend left =0] node[below =0.15 cm] {$$} (D);
\path[-] (A) edge [bend left =0] node[below =0.15 cm] {$$} (C);
\path[-] (R) edge [bend left =0] node[below =0.15 cm] {$$} (E);
\path[-] (I) edge [bend left =0] node[below =0.15 cm] {$$} (F);
\path[-] (D) edge [bend left =0] node[below =0.15 cm] {$$} (N);
\path[-] (C) edge [bend left =0] node[below =0.15 cm] {$$} (B);
\path[-] (E) edge [bend left =0] node[below =0.15 cm] {$$} (H);
\path[-] (F) edge [bend left =0] node[below =0.15 cm] {$$} (G);
\path[-] (N) edge [bend left =0] node[below =0.15 cm] {$$} (O);
\path[-] (B) edge [bend left =0] node[below =0.15 cm] {$$} (P);
\path[-] (H) edge [bend left =0] node[below =0.15 cm] {$$} (P);
\path[-] (G) edge [bend left =0] node[below =0.15 cm] {$$} (O);
\path[-] (P) edge [bend left =0] node[below =0.15 cm] {$$} (Q);
\path[-] (O) edge [bend left =0] node[below =0.15 cm] {$$} (Q);
\path[-] (A) edge [bend left =0] node[below =0.15 cm] {$$} (E);
\path[-] (R) edge [bend left =0] node[below =0.15 cm] {$$} (C);
\path[-] (K) edge [bend left =0] node[below =0.15 cm] {$$} (F);
\path[-] (I) edge [bend left =0] node[below =0.15 cm] {$$} (D);
\path[-] (D) edge [bend left =0] node[below =0.15 cm] {$$} (B);
\path[-] (C) edge [bend left =0] node[below =0.15 cm] {$$} (N);
\path[-] (E) edge [bend left =0] node[below =0.15 cm] {$$} (G);
\path[-] (F) edge [bend left =0] node[below =0.15 cm] {$$} (H);
\end{tikzpicture}
}
\caption{}
\label{dim4}
\end{center}
\end{figure}

\begin{prop}
$\mu (4,9) = 7$.
\end{prop}

\begin{proof}
Let $\mathcal{G}$ be a graph with properties (i) and (ii) and diameter at least $8$.  First let us consider two vertices with maximum distance in $\mathcal{G}$.  If the intersection of these vertices is non-trivial, their distance is bounded above by $\mu (3,8) = 6$.  Thus, the vertices of maximal distance in $\mathcal{G}$ must have trivial intersection.  Call them $ABCD$ and $FGHI$.  
Suppose there exists a $v \in V(\mathcal{G})$ such that $v$ is adjacent to $ABCD$, and $v$ contains $F,G,H$ or $I$.  The shortest path from $v$ to $FGHI$ will be bounded above by $\mu (3,8) =6$, and thus $\mathcal{G}$ will have diameter at most $7$.  Thus no such vertex is contained in $\mathcal{G}$.  Simillarly no vertex containing $A,B,C$ or $D$ adjacent to $FGHI$ is contained in $\mathcal{G}$.

Connectivity of $\mathcal{G}$ requires $\mathcal{G}$ have at least one vertex adjacent to $ABCD$ and at least one vertex adjacent to $FGHI$. These vertices must both contain the only variable which is not in $ABCD$ or $FGHI$, call this variable $E$.  Since $\mathcal{G}$ is locally connected, $\mathcal{G}$ must contain a connected subgraph composed only of the vertices containing $E$.  Every vertex containing $E$ will also contain two variables from either $ABCD$ or $FGHI$.  Take such a vertex, $ABEF$.  Then $\mathcal{G}$ must have a connected subgraph made up of only the vertices containing $AB$.  However, we already have that any vertex adjacent to $ABCD$ must contain $E$.  Thus $ABEF$ must have shortest path length $2$ to $ABCD$.  Any vertex containing $E$ is distance at most $2$ from $ABCD$ or $FGHI$.  Thus $\mathcal{G}$ is at most diameter $5$.  Thus we have a contradiction.  Thus no diameter-$8$ graph with properties (i) and (ii) exists.

To construct a diameter-$7$ graph with properties (i) and (ii), take the graph in Figure \ref{dim4} and append the vertex $EFGI$.
\end{proof}

\begin{prop}
$\mu (4,n) \geq n-d+2$ for all $n \geq 8$.
\end{prop}

\begin{proof}
We can construct an example of a diameter-$(n-d+2)$  graph with propeties (i) and (ii) with $d=4$, $n=8+j$ by taking the graph in Figure \ref{dim4} and appending the vertices \[  EFGx_1 , FGx_1x_2 , \cdots x_{j-3}x_{j-2}x_{j-1}x_j . \]
\end{proof}

\section*{Acknowledgment}
I would like to thank my adviser Hailong Dao for insight and mentoring.  I would like to thank Vic Reiner, Sandra Spiroff, Chris Francisco, Ken Duna, and Bennet Goeckner for insightful conversations.  I would also like to thank Jeremy Martin, Bruno Benedetti, Matteo Varbaro, Margaret Bayer, and my referee for helpful comments on this paper. 

\bibliographystyle{amsalpha}
\bibliography{mybib}
\end{document}